\documentclass[10pt, a4paper]{article}

\usepackage{amssymb, amsmath, amsthm,mathrsfs}
\usepackage{mathbbol, graphicx, latexsym, tkz-graph,mathtools,color}
\usepackage[nobysame]{amsrefs}
\usepackage{cite}

\usepackage[margin = 3 cm]{geometry}

\newtheorem{theorem}{Theorem}[section]
\newtheorem{lemma}[theorem]{Lemma}
	\newtheorem*{theorem*}{Theorem}
\newtheorem{corollary}[theorem]{Corollary}

\theoremstyle{remark}

\usepackage{array}

\begin{document}
\title{Poncelet's Theorem in the four non-isomorphic 
       finite projective planes of order 9}

\author{Norbert Hungerb\"uhler (ETH Z\"urich)\\
           Katharina Kusejko (ETH Z\"urich)}
\date{}

\maketitle

\begin{abstract}\noindent 
We  study Poncelet's  Theorem in  the four  non-isomorphic finite 
projective planes  of order 9. Among these planes, only
the Desarguesian plane turns out to be a Poncelet plane, while
the other three planes which are constructed over the 
miniquaternion near-field of order 9, are not.
This gives  a complete discussion  of Poncelet's  Theorem  
in finite  projective planes  of order 9.
\end{abstract}

\section*{Introduction}

In 1813 Jean-Victor Poncelet \cite{MR1399775} showed one 
of the most beautiful results in projective geometry, known as 
Poncelet's Porism. One version reads as follows.
\begin{theorem*} [Poncelet's Porism]
  Let  $C$ and $C'$ be two conics. 
  If there exists an $m$-sided  polygon, $m \geq 3$, 
  such that the vertices lie on  $C'$ and the sides are
  tangent  to $C$,  then there are  infinitely many  other such  $m$-sided
  polygons. Moreover, for $m \neq m'$, one  cannot find such an  
  $m'$-sided polygon for the same pair of conics $C$ and $C'$.
\end{theorem*}
There are numerous proofs of Poncelet's Theorem in classical geometry
arising from different areas of mathematics:
Synthetic proofs \cite{MR882916},
combinatorial proofs \cite{LHNH} and purely geometric proofs
using properties of the Euclidean plane.
Moreover, a deep connection between Poncelet's Porism and the theory of elliptic
curves has been established \cite{MR497281}. 
See the recent book by V. Dragovi{\'c} and M. Radnovi{\'c} \cite{DR2011} 
for an overview.

In the present paper, we consider Poncelet's Theorem in finite geometries.
In particular, we introduce the notion of a \emph{Poncelet plane} in order to restate Poncelet's Theorem
for finite projective planes. For $q:=p^k$, $p$ a prime and $k \geq 1$, 
it is well-known how to construct a finite projective plane over the finite field $GF(q)$ of order $q$, 
denoted by $PG(2,q)$ and also known as finite projective Desargues\-ian plane.
Many properties of the real projective plane carry over to $PG(2,q)$.
In~\cite[Section 16.6]{MR882916}, Berger presented a proof of
the general form of Poncelet's Theorem, the {\em Great Poncelet Theorem}, for
projective planes over general fields with more than five elements.
But there exist finite projective planes which are not isomorphic to a projective Desarguesian plane $PG(2,q)$.
The smallest order where one can find such examples is the order 9.
In particular, there are exactly four non-isomorphic finite projective planes of order 9, as proved by Lam et al.\ in \cite{MR1140586}.
Besides $PG(2,9)$ there are three non-Desarguesian finite projective planes of order 9,
all of them constructed over a near-field of order 9. And our main result reads as follows.
\begin{theorem*}
  The only Poncelet plane of order 9 is the finite projective Desarguesian plane $PG(2,9)$.
\end{theorem*}
The main objects when studying Poncelet's Theorem in finite projective planes are ovals, 
which are a generalization of conics.
One difficulty when working in planes not constructed over a finite field,
such as the three planes of order 9 we consider in this paper,
is to find such ovals.
Since order 9 is the smallest one where the question of finding 
ovals in non-Desarguesian planes becomes important,
some work has been done on ovals (and generalizations thereof
like unitals and arcs) in planes of this order.
For some recent work, see for example \cite{MR2501969, 
MR2854273, MR2528997, MR2254659, MR2906536, MR1936288}.
\section{Preliminaries}
We briefly recall   the  basic  definitions
concerning finite projective planes (see e.g. \cite{MR1612570}).

The   triple  $(\mathbb{P},\mathbb{B},\mathbb{I})$
with  $\mathbb{I} \subset  \mathbb{P} \times  \mathbb{B}$ is  called
\emph{projective plane}, if the following axioms are satisfied.
\begin{enumerate}\itemsep=0mm
\item[(A1)] For  any two elements  $P,Q \in  \mathbb{P}$, $P \neq  Q$, there
  exists  a  unique  element  $g   \in  \mathbb{B}$  with  $(P,g)  \in
  \mathbb{I}$ and $(Q,g) \in \mathbb{I}$.
\item[(A2)] For  any two elements  $g,h \in  \mathbb{B}$, $g \neq  h$, there
  exists  a  unique  element  $P   \in  \mathbb{P}$  with  $(P,g)  \in
  \mathbb{I}$ and $(P,h) \in \mathbb{I}$.
\item[(A3)] There are four  elements $P_1,\ldots,P_4\in\mathbb{P}$ such that
  $\forall   g\in\mathbb{B}$  we   have   $(P_i,g)\in\mathbb  I$   and
  $(P_j,g)\in\mathbb I$  with $i\neq j$  implies $(P_k,g)\notin\mathbb
  I$ for $k\neq i,j$.
\end{enumerate}

Elements   of  $\mathbb{P}$   are  called   \emph{points}  and   elements  of
$\mathbb{B}$ are called \emph{lines}. 
By $(P,g) \in \mathbb{I}$ we denote that the
point $P$ is incident with the line $g$. 
A more convenient notation of
this incidence  relation is  $P \in  g$. 
Three points  are said  to be
\emph{collinear} if they are incident with  the same line and three lines are
said to  be \emph{concurrent} if  they are incident  with the same  point.  
A projective  plane  is called  \emph{finite},  if  the sets  $\mathbb{P}$  and
$\mathbb{B}$ are finite. 
In that case,  it turns out that each line is
incident  with $n+1$ points and  each point  is
incident with $n+1$ of lines, for some $n \geq 1$. 
According to that,  a   finite        projective       plane
  $(\mathbb{P},\mathbb{B},\mathbb{I})$  is said  to be  of \emph{order
    $n$},  if   $\  \left|\mathbb{P}\right|=\left|\mathbb{B}\right|  =
  n^2+n+1$, and denoted by $\mathcal{P}_n$.

For any finite projective plane  $\mathcal{P}_n =(\mathbb{P},\mathbb{B},\mathbb{I})$ of
  order             $n$, we define the  \emph{dual
    projective plane}  of $\mathcal{P}_n$ by $\mathcal{P}^D_n
  :=(\mathbb{B},\mathbb{P},\mathbb{I^*}$),  with $(P,g)  \in \mathbb{I}
  \Leftrightarrow  (g,P) \in  \mathbb{I^*}$.
  Then $\mathcal{P}_n$  is called
  \emph{self-dual}, if $\mathcal{P}_n \cong \mathcal{P}^D_n $, i.e.\ if
  there  exists   a  bijective   map  $\phi:(\mathbb{P},\mathbb{B})\to
  (\mathbb{B},\mathbb{P})$    such    that   $(P,g)\in\mathbb    I\iff
  \phi(P,g)\in\mathbb  I^*$.   
  Incidence  state\-ments  where  the  sets
  $\mathbb{P}$  and  $\mathbb{B}$  are  interchanged are  said  to  be
  \emph{dual to each other}.

In  this paper,  we are  mainly  interested in finite  projective
planes of  order 9, which means  that we have $9^2+9+1=91$  points and
$91$  lines. Each  line  is incident  with 10  points and each point  is
incident  with 10  lines.  

In  order to  generalize  conics to  finite
projective planes, the notion of ovals has been introduced:
An \emph{oval} in $\mathcal{P}_n$ is a set of $n+1$
points, no three of which are collinear.
Every conic is an oval, and for $p$ odd, every oval in $PG(2,p^k)$  is a conic.

A line which intersects an oval $O$ in  two points is called \emph{secant},  a line which
intersects $O$ in  one point is called \emph{tangent}, and a line which is
disjoint to $O$ is called \emph{external line} of $O$.

To reformulate Poncelet's Theorem for finite projective planes,
we take a closer look at pairs of ovals $O_t$   and   $O_s$  in $\mathcal{P}_n$.
An \emph{$m$-sided  Poncelet  polygon} is  a  polygon with  $m$
  sides, $3 \leq m \leq n+1$,  such that the vertices are on $O_s$ and the sides are tangent to $O_t$. 
  According to that, we call  $O_t$ the  \emph{tangent oval}  and $O_s$  the
  \emph{secant oval} of the Poncelet polygon.
 For $3 \leq m \leq n+1$ fixed, $(O_t,O_s)$ is said to form a
  \emph{Poncelet  $m$-pair}, if  there exists  at least  one $m$-sided
  Poncelet polygon  for $O_t$  and $O_s$,  but no  $m'$-sided Poncelet
  polygon, $m' \neq m,\ 3 \leq m'  \leq n+1$, for the same pair can be
  constructed.
 We  say that  $(O_t,O_s)$ forms  a \emph{Poncelet  $0$-pair}, if
   no secant of $O_s$ is a tangent of $O_t$.
 We  say that $(O_t,O_s)$ forms  a \emph{Poncelet $\infty$-pair},
  if there exists at least one secant  of $O_s$, which is a tangent of
  $O_t$, but no $m$-sided Poncelet polygon for $3 \leq m \leq n+1$ can
  be constructed.

Note that finite projective planes may exhibit
  rather unintuitive phenomena compared  to the real projective plane.
  For example, a pair of ovals can be  located such that no point of one oval
  is incident with a  tangent of the other one and  vice versa. 
  Or, in finite projective planes  of even order it may happen  that two ovals
  have all their tangents in common.

With the terminology above, we call a   finite  projective
  plane $\mathcal{P}_n$ a \emph{Poncelet plane}, if  every pair of
  ovals $(O_t,O_s)$ is  a Poncelet $m$-pair, for $3 \leq  m \leq n+1$,
  $m=0$ or $m= \infty$.

Note that in planes of even order all tangents  of an oval meet in one point,
  the so-called nucleus  of the  oval.  Because of that,  only Poncelet  $0$-pairs and
  Poncelet $\infty$-pairs can be constructed  in such planes. In this sense,
  all planes of even order are Poncelet planes.
\section{Poncelet's Theorem in $PG(2,9)$}
The main goal in this section is to show that $PG(2,9)$ is a Poncelet plane.
As mentioned earlier, Berger presented in \cite{MR882916} a synthetic proof of
the \emph{Great Poncelet Theorem}.
His proof is formulated for projective planes over an arbitrary field
with at least five elements. However,  
a number of additional thoughts are necessary to 
ensure that all steps in the proof work out over fields which are not algebraically closed.
Since Berger shows a more general version of Poncelet's Theorem
we want to avoid this discussion, and,
in order to make the paper self-contained, 
we present a shorter proof for Poncelet's Theorem in $PG(2,9)$
which is based upon Pascal's Theorem, similar to~\cite{LHNH}. However, 
as we work in a finite plane, we employ combinatorial arguments
in a completely different way compared to~\cite{LHNH}.
In particular, we show the following statement.
\begin{theorem}
  The finite projective Desarguesian plane of order 9 is a Poncelet plane.
\end{theorem}\vspace*{-1mm}
In  planes of order  9, ovals consist of  10 points. To prove that
$PG(2,9)$ is a  Poncelet plane, it is therefore enough  to show that
if a $3$-sided  or a $4$-sided Poncelet polygon exists for a pair
of ovals $(O_t,O_s)$,  then this pair is a Poncelet  $3$-pair or $4$-pair,
respectively.

Let us quickly recall the construction of the finite projective Desarguesian plane $PG(2,q)$ constructed over $GF(q)$, 
$q:=p^k$, $p$ an odd prime and $k \geq 1$. 
The points of $PG(2,q)$ are given by non-zero column vectors $[x,y,z]^T$ for $x,y,z \in GF(q)$, 
where $[\lambda x,\lambda y,\lambda z] = [x,y,z]$ for all $\lambda \in GF(q)\setminus \{0\}$. 
Similarly, all lines are denoted by row vectors $[x,y,z]$.
A point $[x,y,z]^T$ is incident with a line $[a,b,c]$ if $ax+by+cz=0$ in $GF(q)$.

The  following  facts are  a  collection of some elementary
properties  we  will  use  later  on (see e.g.
 \cite{MR1612570} for proofs).
\begin{lemma}  \label{connection}
  Let  $g=[g_1,g_2,g_3]$  and  $h=[h_1,h_2,h_3]$  be  two
  different lines in  $PG(2,q)$. The unique intersection  point $P$ of
  $g$ and $h$ is given by the vector product of $g$ and $h$, i.e.\
    $$ P = [g_2h_3-g_3h_2,\ g_3h_1-g_1h_3,\ g_1h_2-g_2h_1]^T. $$
  Similarly, for two points  $P =  [P_1,P_2,P_3]^T$ and  $Q =
  [Q_1,Q_2,Q_3]^T$ in $PG(2,q)$, the unique line
  $g$ through $P$  and $Q$ is given  by
    $$ g = [P_2Q_3-P_3Q_2,\ P_3Q_1-P_1Q_3,\ P_1Q_2-P_2Q_1]. $$
\end{lemma}
\begin{lemma} \label{CollDet}   
  Let    $P=[P_1,P_2,P_3]^T$,   $Q    =
  [Q_1,Q_2,Q_3]^T$ and $R = [R_1,R_2,R_3]^T$  be three points in
  $PG(2,q)$. 
  Then $ P,Q,R$ are collinear if and only if
$$\det\begin{pmatrix} P_1 & Q_1 & R_1 \\ P_2 & Q_2 & R_2 \\ P_3 & Q_3 & R_3  \end{pmatrix}=0. $$
\end{lemma}\vspace*{-2mm}
In finite projective Desarguesian planes $PG(2,q)$ over
a field of odd characteristic,  ovals coincide with conics (see, e.g., \cite{MR0344216}).  
Thus, an oval can be described as the solutions of
\begin{equation} \label{eq ovals}
O: ax^2+by^2+cz^2+dxy+exz+fyz=0,
\end{equation}
where $a,b,c,d,e,f \in GF(q),\ (a,b,c,d,e,f) \neq (0,0,0,0,0,0)$ and the
matrix $M_O$ associated to this quadratic form,
$$ M_O = \begin{pmatrix} a & d/2 & e/2 \\ 
d/2 & b & f/2 \\ e/2 & f/2 & c \end{pmatrix},$$
is non-singular for ovals. 
Otherwise, for $M_O$ singular, the equation \eqref{eq ovals} describes a line, a pair of lines or a point.

The next step is to show the following result for Poncelet triangles.
\begin{theorem}\label{PonceletTriangle}
  Let $(O_t,O_s)$ be a pair of ovals in $PG(2,q)$ such that a Poncelet triangle can be
  constructed. Then  no $m$-sided Poncelet polygon  for $4 \leq m  \leq q+1$
  for the same pair of ovals exists.
\end{theorem}
To see this, we need some preliminary results.
\begin{theorem}  
  Let  $A,B,C,D,E$   and  $F$  be the six vertices of a
   hexagon such that no three of them are collinear.
  Then, the intersection points of opposite sides
    $$ P=\overline{AB} \cap \overline{DE}, \ Q = \overline{BC} \cap \overline{EF}, \ R=\overline{CD} \cap \overline{AF} $$
  are collinear if and only if
  the points $A,B,C,D,E$ and $F$ lie on an oval.
\label{Pascalreverse}
\end{theorem}
The if-statement is Pascal's Theorem, the converse is known as the  
 Braikenridge--Maclaurin Theorem. The line through $P,Q,R$ is called
Pascal's line (see Figure \ref{fig:Pascal}).

\begin{figure}[h]
  \begin{center}
\definecolor{ffqqqq}{rgb}{1,0,0}
\definecolor{qqqqff}{rgb}{0,0,1}
\begin{tikzpicture}[line cap=round,line join=round,>=triangle 45,x=0.7cm,y=0.7cm]
\clip(-3.61,-1.96) rectangle (5.83,4.5);
\draw [rotate around={-1.1:(1.14,1.22)},color=qqqqff] (1.14,1.22) ellipse (2.84cm and 1.81cm);
\draw [domain=-3.61:5.83] plot(\x,{(--8.5--3.96*\x)/3.23});
\draw [domain=-3.61:5.83] plot(\x,{(--16.26-4.38*\x)/3.16});
\draw [domain=-3.61:5.83] plot(\x,{(--0.51--4.33*\x)/-2.57});
\draw [domain=-3.61:5.83] plot(\x,{(-8.01--3.76*\x)/4.03});
\draw [domain=-3.61:5.83] plot(\x,{(-4.77-2.56*\x)/-6.98});
\draw [domain=-3.61:5.83] plot(\x,{(--11.21-3.55*\x)/6.01});
\draw [color=ffqqqq,domain=-3.61:5.83] plot(\x,{(--2.88-0.31*\x)/2.21});
\begin{scriptsize}
\fill [color=qqqqff] (-2.27,-0.15) circle (1.5pt);
\draw[color=qqqqff] (-2.7,-0.1) node {$A$};
\fill [color=qqqqff] (0.96,3.81) circle (1.5pt);
\draw[color=qqqqff] (0.45,4) node {$B$};
\fill [color=qqqqff] (4.12,-0.57) circle (1.5pt);
\draw[color=qqqqff] (4.07,-.9) node {$C$};
\fill [color=qqqqff] (0.68,-1.35) circle (1.5pt);
\draw[color=qqqqff] (0.2,-1.5) node {$E$};
\fill [color=qqqqff] (-1.89,2.98) circle (1.5pt);
\draw[color=qqqqff] (-2.4,2.85) node {$D$};
\fill [color=qqqqff] (4.71,2.41) circle (1.5pt);
\draw[color=qqqqff] (5.13,2.3) node {$F$};
\fill [color=ffqqqq] (-0.97,1.44) circle (1.5pt);
\draw[color=ffqqqq] (-1.4,1.2) node {$P$};
\fill [color=ffqqqq] (1.23,1.14) circle (1.5pt);
\draw[color=ffqqqq] (1.25,0.73) node {$R$};
\fill [color=ffqqqq] (3.07,0.88) circle (1.5pt);
\draw[color=ffqqqq] (3.66,0.5) node {$Q$};
\end{scriptsize}
\end{tikzpicture}
    \caption{Pascal's Theorem}
    \label{fig:Pascal}
  \end{center}
\end{figure}
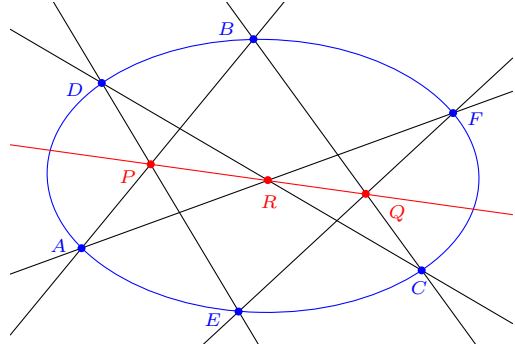
\begin{proof}
  We choose coordinates such that
$$ A = [1,0,0]^T,\ C=[0,1,0]^T,\ E=[0,0,1]^T. $$
Since no three of the points $A,B,C,D,E$ and $F$ are collinear, 
all coordinates
of the remaining three points are non-zero, so we have
$$ B = [1,B_2,B_3]^T,\ D=[1,D_2,D_3]^T,\ F=[1,F_2,F_3]^T $$
with  $B_2,B_3,D_2,D_3,F_2,F_3 \neq  0$. 
Moreover,  we have  $B_2 \neq
D_2$, $D_2 \neq  F_2$, $B_2 \neq F_2$, $B_3 \neq  D_3$, $D_3 \neq F_3$
and $B_3 \neq F_3$. To see this, assume $B_2 = D_2$. In this case, the
points
$$ B = [1,B_2,B_3]^T,\ D=[1,B_2,D_3]^T,\ E=[0,0,1]^T$$
would be  collinear,  since they  are  all  incident  with  the line  $g  =
[-B_2,1,0]$. The  same can be
shown  analogously  for  the  other  coordinates.   Using  Lemma \ref{connection}, 
we get the connecting lines
\begin{align*} 
\overline{AF} &= [0, F_3, -F_2],& \overline{AB} &=[0, B_3, -B_2],& \overline{BC} &= [B_3, 0, -1],\\
\overline{CD} &= [D_3, 0, -1],  & \overline{DE} &= [D_2, -1, 0],  & \overline{EF} &= [F_2, -1, 0].
\end{align*}
Using Lemma \ref{connection} once more, we obtain
$$ P=[B_2, B_2 D_2, B_3 D_2]^T,\ Q = [1, F_2, B_3]^T,\ R = [F_3, D_3 F_2, D_3 F_3]^T.$$
Obviously, $F$ lies on the unique conic through the points $A,B,C,D,E$ iff
\begin{eqnarray*}
\lefteqn{\hspace*{-2mm}
\det\begin{pmatrix}
A_1^2 & A_2^2 & A_3^2 & A_1 A_2 & A_1 A_3 & A_2 A_3 \\
C_1^2 & C_2^2 & C_3^2 & C_1 C_2 & C_1 C_3 & C_2 C_3 \\
E_1^2 & E_2^2 & E_3^2 & E_1 E_2 & E_1 E_3 & E_2 E_3 \\
B_1^2 & B_2^2 & B_3^2 & B_1 B_2 & B_1 B_3 & B_2 B_3 \\
D_1^2 & D_2^2 & D_3^2 & D_1 D_2 & D_1 D_3 & D_2 D_3 \\
F_1^2 & F_2^2 & F_3^2 & F_1 F_2 & F_1 F_3 & F_2 F_3 
\end{pmatrix}}\\
&=&
\det\begin{pmatrix}
1 & 0 & 0 & 0 & 0 & 0 \\
0 & 1 & 0 & 0 & 0 & 0 \\
0 & 0 & 1 & 0 & 0 & 0 \\
1 & B_2^2 & B_3^2 & B_2 & B_3 & B_2 B_3 \\
1 & D_2^2 & D_3^2 & D_2 & D_3 & D_2 D_3 \\
1 & F_2^2 & F_3^2 & F_2 & F_3 & F_2 F_3 
\end{pmatrix}\\
&=&\det\begin{pmatrix}
B_2 & B_3 & B_2 B_3 \\
D_2 & D_3 & D_2 D_3 \\
F_2 & F_3 & F_2 F_3 
\end{pmatrix} = \det (P,Q,R)=0
\end{eqnarray*}
But according to Lemma~\ref{CollDet}, this is precisely the
case for $P,Q$ and $R$ being collinear.
\end{proof}
Note that all finite projective  Desarguesian planes are self-dual (see~\cite{MR1612570}) and hence,
we may consider the dual form of Theorem~\ref{Pascalreverse}.
\begin{corollary}  
  Let  $A,B,C,D,E$  and  $F$   be  the six  vertices of a
  hexagon such that no three of them are collinear.  Then,  the diagonals  $\overline{AD}$,
  $\overline{BE}$ and $\overline{CF}$ meet in one point, the 
  Brianchon point,   if and only if the sides
  $\overline{AB}$, $\overline{BC}$,  $\overline{CD}$, $\overline{DE}$,
  $\overline{EF}$ and $\overline{FA}$ are tangents of an oval.
\label{Brianchonconverse}
\end{corollary}
\begin{lemma}
  Let $O_s$ be an oval with  two inscribed triangles $\triangle ACE$
  and $\triangle BDF$, such  that no three of the vertices  are collinear.  Then
  the sides of the two triangles
  are tangents of an oval $O_t$.
\label{inscribedtriangles}
\end{lemma}
This result was used in \cite{Rohn} to prove Poncelet's Theorem in the
real projective  plane for triangles.  Since the arguments  used there
cannot be applied  to the finite projective plane, we  have to give an
alternative proof here.

{\em Proof of Lemma  \ref{inscribedtriangles}.} 
Let $ A = [1,0,0]^T$, $C=[0,1,0]^T$ and $E=[0,0,1]^T$
be on $O_s$, which leads to the oval equation $ x y + e x z+ f y z = 0$,
with $e  \neq 0$ and  $f \neq  0$. 
For every other point of this oval, all
three coordinates are  non-zero. In particular, we have  (by scaling if
necessary)
$$ B = [1,B_2,B_3]^T,\ D=[1,D_2,D_3]^T,\ F=[1,F_2,F_3]^T $$
with $B_2, B_3,  D_2, D_3, F_2$ and $F_3$ non-zero.   The sides of 
$\triangle ACE$ and $\triangle BDF$ are denoted by
\begin{align*}
& \triangle ACE:\ g_1 = \overline{AC},\  \,g_3 = \overline{CE},\ g_5 = \overline{EA},\\ 
& \triangle BDF:\ g_2 = \overline{BD},\  g_4 = \overline{DF},\ g_6 = \overline{FB}.
\end{align*}
Explicitly, we have
\begin{align*}
g_1 &= [0, 0, 1],\ \ g_2 = [-B_3 D_2 + B_2 D_3, B_3 - D_3, -B_2 + D_2]\\
g_3 &= [1, 0, 0],\ \ g_4 = [-D_3 F_2 + D_2 F_3, D_3 - F_3, -D_2 + F_2]\\
g_5 &= [0, 1, 0],\ \ g_6 = [B_3 F_2 - B_2 F_3, -B_3 + F_3, B_2 - F_2].
\end{align*}
The intersection points of these lines are given by
\begin{align*}
A_1 &= g_6 \cap g_1,& A_2 &= g_1 \cap g_2,& A_3 &= g_2 \cap g_3, \\
A_4 &= g_3 \cap g_4,& A_5 &= g_4 \cap g_5,& A_6 &= g_5 \cap g_6. 
\end{align*}
This leads to
\begin{align*}
 A_1 &= [B_3 - F_3, B_3 F_2 - B_2 F_3, 0]^T\!\!, \hspace*{-1mm}& A_2 &= [B_3 - D_3, B_3 D_2 - B_2 D_3, 0]^T\!\!, \\
 A_3 &= [0, B_2 - D_2, B_3 - D_3]^T\!\!,              \hspace*{-1mm}& A_4 &= [0, D_2 - F_2, D_3 - F_3]^T\!\!, \\
 A_5 &= [D_2 - F_2, 0, -D_3 F_2 + D_2 F_3]^T\!\!, \hspace*{-1mm}& A_6 &= [B_2 - F_2, 0, -B_3 F_2 + B_2 F_3]^T\!\!.
\end{align*}
We would  like to find  an oval $O_t$, such  that the lines  $g_1,\ldots,g_6$ are
tangents  of  it.     By     Brianchon's Theorem     (Corollary
\ref{Brianchonconverse}), we  know that this is  equivalent to showing
that $ \overline{A_1A_4}$, $\overline{A_2A_5}$ and $\overline{A_3A_6}$
meet in one point.  Using Lemma \ref{connection}, we obtain
\begin{align*}
 \overline{A_1A_4} &= [B_3 D_3 F_2 - B_2 D_3 F_3 - B_3 F_2 F_3 + B_2 F_3^2, \\
&\,\quad -B_3 D_3 + B_3 F_3 + D_3 F_3 - F_3^2, B_3 D_2 - B_3 F_2 - D_2 F_3 + F_2 F_3] \\[2mm]
\overline{A_2A_5} &= [B_3 D_2 D_3 F_2 - B_2 D_3^2 F_2 - B_3 D_2^2 F_3 + B_2 D_2 D_3 F_3,\\
&\,\quad -B_3 D_3 F_2 + D_3^2 F_2 + B_3 D_2 F_3 - D_2 D_3 F_3, B_3 D_2^2 - B_2 D_2 D_3 - B_3 D_2 F_2 + B_2 D_3 F_2] \\[2mm]
\overline{A_3A_6} &= [B_2 B_3 F_2 - B_3 D_2 F_2 - B_2^2 F_3 + B_2 D_2 F_3,\\
 &\,\quad -B_2 B_3 + B_2 D_3 + B_3 F_2 - D_3 F_2, B_2^2 - B_2 D_2 - B_2 F_2 + D_2 F_2].
\end{align*}
Observe that the points $B,D$ and $F$ lie on the original oval, which means that they satisfy
$$ B_2 = \frac{- e B_3}{1 + f B_3},\ D_2 = \frac{- e D_3}{1 + f D_3}, \ F_2 = \frac{- e F_3}{1 + f F_3}.$$
Note that we have  $1 + f B_3 \neq 0$,  $1 + f D_3 \neq 0$  and $1 + f
F_3 \neq 0$. To see this, assume $1 +  f B_3 = 0$.  It follows $ B_3 =
-\frac{1}{f}$ and using  the oval equation $ x y  + e x z+ f y  z = 0$
once more, we  obtain $-\frac{e}{f} = 0$, contradicting  the fact that
$e\neq 0$ and  $ f \neq 0$.   Finally it follows that  the three lines
are concurrent, since
\begin{equation}\det(\overline{A_1A_4},\overline{A_2A_5},\overline{A_3A_6}) = 0.\tag*{$\Box$}
\end{equation}
\begin{proof}[Proof    of    Theorem   \ref{PonceletTriangle}.]    Let
  $(O_t,O_s)$ be  a pair of ovals,  such that there exists  a triangle
  which consists  of tangents of $O_t$  and  vertices  on $O_s$.  Let the
  sides of  this triangle be $t_1,t_2$ and  $t_3$.  Assume
  that there exists another closed polygon using tangents of $O_t$ with vertices
  on $O_s$. Since we need at least three vertices on $O_s$, such that the lines connecting these
  are tangent  to $O_t$ for the new polygon,  we can assume the
  existence of at least three such lines. Hence, we start with $s_1$,
  which  is a  tangent  of $O_t$  and  joins two points  of  $O_s$, denoted by $S_1$ and $S_2$.  By assumption,
  there exists  another line $s_2$, which joins $S_2$ with another point of  $O_s$ and which is a
  tangent of $O_t$. Let $s_2$ intersect $O_s$ in $S_2$ and $S_3$.
  The claim is now, that the line connecting $S_1$ and $S_3$, denoted by
$s_3$, is  a tangent of  $O_t$ as  well.  Assume the  contrary (Figure
\ref{fig:Opposite}), i.e.~assume that $s_3$ is not a tangent of $O_t$.
By Lemma  \ref{inscribedtriangles}, there  exists an oval $\tilde{O}$
such   that   $t_1,t_2,t_3,s_1,s_2$   and  $s_3$   are   tangents   of
$\tilde{O}$. But $t_1,t_2,t_3,s_1$ and $s_2$  are tangents of $O_t$ as
well. Since an oval is uniquely determined by five of its tangents, we
have $O_t$ =  $\tilde{O}$. Hence, every other polygon which can  be closed is
a  Poncelet  triangle  as  well. This  shows  that
$(O_t,O_s)$ is a Poncelet 3-pair.
\end{proof}
\begin{figure}[h]
  \begin{center}
\definecolor{uuuuuu}{rgb}{0.27,0.27,0.27}
\definecolor{ffqqqq}{rgb}{1,0,0}
\definecolor{qqwuqq}{rgb}{0,0.39,0}
\begin{tikzpicture}[line cap=round,line join=round,>=triangle 45,x=0.7cm,y=0.7cm]
\clip(-8.05,-2.67) rectangle (5.89,3.7);
\draw [rotate around={23.71:(-1.42,0.57)}] (-1.42,0.57) ellipse (1.58cm and 0.8cm);
\draw [color=qqwuqq,domain=-8.05:5.89] plot(\x,{(-7.24-2.02*\x)/-0.95});
\draw [color=qqwuqq,domain=-8.05:5.89] plot(\x,{(--4.63-0.4*\x)/2.41});
\draw [color=qqwuqq,domain=-8.05:5.89] plot(\x,{(-0.39--1.23*\x)/3.26});
\draw [color=ffqqqq,domain=-8.05:5.89] plot(\x,{(-7.85-1.98*\x)/2.31});
\draw [color=ffqqqq,domain=-8.05:5.89] plot(\x,{(--9.92--0.52*\x)/4.91});
\draw [rotate around={-173.77:(-1.65,0.31)}] (-1.65,0.31) ellipse (3.9cm and 1.48cm);
\draw [color=ffqqqq,domain=-0.3921704258884361:5.890997940382233] plot(\x,{(-1.32--3.31*\x)/2.44});
\draw [color=ffqqqq,domain=-8.054976290573041:-0.3921704258884361] plot(\x,{(--1.24-0.77*\x)/-1.43});
\begin{scriptsize}
\draw[color=black] (-2.44,0.94) node {$O_t$};
\draw[color=qqwuqq] (-1.67,3.33) node {$t_1$};
\draw[color=qqwuqq] (-6.36,3.22) node {$t_2$};
\draw[color=qqwuqq] (-6.03,-2.1) node {$t_3$};
\fill [color=black] (-5.63,1.42) circle (1.5pt);
\draw[color=black] (-5.5,1.74) node {$S_2$};
\draw[color=ffqqqq] (-6.84,2.19) node {$s_1$};
\draw[color=ffqqqq] (-7.5,0.92) node {$s_2$};
\fill [color=black] (2.05,2.24) circle (1.5pt);
\draw[color=black] (1.9,2.52) node {$S_3$};
\draw[color=black] (2.4,-1.22) node {$O_s$};
\draw[color=ffqqqq] (1.35,0.9) node {$s_3$};
\fill [color=uuuuuu] (-1.82,-1.84) circle (1.5pt);
\draw[color=uuuuuu] (-1.76,-2.2) node {$S_1$};
\end{scriptsize}
\end{tikzpicture}
    \caption{Opposite assumption: Assume that the line connecting $S_1$ and $S_3$ is not tangent to $O_t$.}
    \label{fig:Opposite}
  \end{center}
\end{figure}
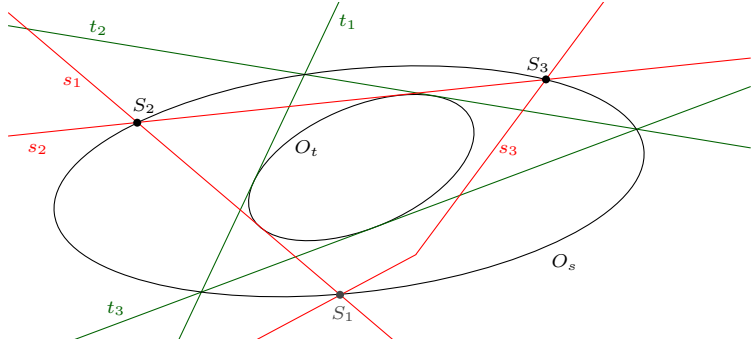
Note that we did not restrict ourselves to $PG(2,9)$ in the above proof, i.e.\
Theorem \ref{PonceletTriangle} applies to all finite projective Desarguesian planes $PG(2,q)$.

Now we turn our attention to $4$-sided Poncelet polygons.
\begin{theorem}
  Let $(O_t,O_s)$ be a pair of ovals in $PG(2,9)$ which carries a Poncelet quadrilateral. 
  Then  no $m$-sided Poncelet polygon, for $3 \leq m  \leq 10$, $m\neq 4$,  
  for $(O_t,O_s)$ can be constructed.
\label{PonceletQuadrilateral}
\end{theorem}
We need to show that  the  existence of  a
Poncelet  quadrilateral for a  pair of ovals $(O_t,O_s)$ excludes
the existence of a $5$-sided Poncelet polygon as well as the existence
of a  $6$-sided Poncelet polygon  for the same  pair. To see  this, we
start  with a  pair  $(O_t,O_s)$ which  carries  a Poncelet quadrilateral.   
Recall   the  following  fundamental  theorem   for  $PG(2,q)$ (see~\cite{MR1612570}).
\begin{theorem}  
Let $\left\{  P_1,P_2,P_3,P_4 \right\}$  and $\left\{
    Q_1,Q_2,Q_3,Q_4 \right\}$  be sets  of four  points, such  that no
  three points  of the same  set are  collinear.  Then there  exists a
  unique projective map $T$, such  that $T(P_i)=Q_i$, for $1\leq i \leq 4$.
\end{theorem}
{\em Proof of Theorem~\ref{PonceletQuadrilateral}.} 
We may assume that  the pair of ovals  $(O_t,O_s)$ carries
the Poncelet quadrilateral
$$ A=[1,-1,0]^T, \ B=[1,0,-1]^T,\ C=[1,1,0]^T,\ D=[1,0,1]^T. $$
The equation of $O_s$ is of the  form
\begin{equation}
\label{generalOval}
a x^2+b y^2 + c z^2 + d x y + e x z + f y z =0,
\end{equation}
for $a,b,c,d,e,f \in GF(9)$
and the associated matrix is non-singular.  We want the points $A,B,C,D$ to
lie on the  oval $O_s$, which gives four conditions for
(\ref{generalOval}), namely
$$ 
a+b-d = 0,\quad 
a+c-e =0, \quad
a+b+d =0, \quad
a+c+e =0.
$$
Since $a=0$ would lead to a singular matrix we may scale $a=1$ and the equation for $O_s$ is
\begin{align}
O_s(f): x^2-y^2-z^2+2f yz = 0,
\end{align}
for $f\neq  \pm 1$, which  ensures that  the associated  matrix is
non-singular. It  is enough  to consider  ovals of the  above form  for $O_s$. 
Now, the four lines
$$ \overline{AB}=[1,1,1],\ \overline{BC}=[1,-1,1],\ \overline{CD}=[1,-1,-1],\ \overline{DA}=[1,1,-1] $$
need to be tangents of $O_t$.
To find the corresponding oval equation, we first determine the equations
of      ovals      which      contain      the       four      points
$[1,1,1]^T$, $[1,-1,1]^T$, $[1,-1,-1]^T$ and $[1,1,-1]^T$. We have to solve
the system of equations for its coefficients
\begin{align*} 
a+b+c+d+e+2f &= 0,\quad
a+b+c-d+e-2f =0,\\
a+b+c-d-e+2f &=0,\quad
a+b+c+d-e-2f =0.
\end{align*}
We immediately  obtain $d=e=f=0$  and after scaling  $a=1$, we  end up
with the equation $x^2  + b y^2 -(1+b) z^2=0$. Since  we need an oval
with   tangents   $[1,1,1]$, $[1,-1,1]$, $[1,-1,-1]$ and $[1,1,-1]$
rather than points, we have to  take the equation which corresponds to
the inverse matrix  of the matrix associated to the  equation $x^2 + b
y^2 -(1+b) z^2=0  $ which is
\begin{align}
O_t(b): x^2 + \frac{1}{b} y^2 - \frac{1}{1+b} z^2 =0,
\end{align}
for $b \neq 0,-1$. 
To exclude the simultaneous existence  of a $4$-sided Poncelet polygon
and a  $5$-sided or  $6$-sided Poncelet  polygon, respectively,  it is
enough  to   consider   pairs  of   the ovals  described
above. Before  we start  analyzing these oval  pairs, we  collect some
facts about the field we are  working in. Note that the multiplicative
group of a finite field is cyclic, hence we can write $GF(9)$ as
$$ GF(9)=\left\{0,1,a,a^2,a^3,a^4,a^5,a^6,a^7 \right\}, $$
where  $a$  is a  root  of  the  polynomial $f(x)=x^2+x-1$,  which  is
irreducible over $GF(3)$. Addition and multiplication obey the rules
$$ a^2+a=1,\ a^i a^j = a^{i+j},\ a^8=1. $$
Therefore, we consider pairs of ovals of the form $(O_t(b),O_s(f))$ for
$$ b \in \left\{1,a,a^2,a^3,a^5,a^6,a^7\right\} \text{ and } f \in \left\{0,a,a^2,a^3,a^5,a^6,a^7\right\}. $$
By inspecting the pair  $(O_t(b),O_s(f))$ we obtain exactly the
  same results as for $(O_t(b),O_s(-f))$, because changing the sign
  of the $y$-coordinate has the effect
  $$ [x,y,z]^T\in O_s(f) \iff [x,-y,z]^T \in O_s(-f)$$
 and 
  $$[x,y,z]^T\in O_t(b) \iff [x,-y,z]^T \in O_t(b).$$
  Hence, it     is     enough    to     consider     $f     \in
  \left\{0,a,a^2,a^3\right\}$.
  
Moreover, when calculating the coefficients $\frac{1}{b}$ and $\frac{1}{1+b}$ for all values of $b$ above, we obtain
  \begin{align*} 
   O_t(1)&: x^2+y^2+z^2=0                    & O_t(a^5)&: x^2 + a^3 y^2 + a^2 z^2 =0 \\
   O_t(a)&: x^2 + a^7 y^2 + a^5 z^2 =0 &   O_t(a^6)&: x^2 + a^2 y^2 + a^3 z^2 =0 \\
   O_t(a^2)&: x^2 + a^6 y^2 + a z^2 =0   & O_t(a^7)&: x^2 + a y^2 + a^6 z^2 =0.\\
   O_t(a^3)&: x^2 + a^5 y^2 + a^7 z^2 =0 
\end{align*}
Note that interchanging the $y$ and $z$ coordinate does not change the
incidence relations for both ovals, as both equations are symmetric, i.e.\
$$ [x,y,z]^T \in O_t(b) \Leftrightarrow [x,z,y]^T \in O_t(b)$$
and
$$[x,y,z]^T \in O_s(f) \Leftrightarrow [x,z,y]^T \in O_s(f). $$
Therefore, it is  enough to consider $b \in \left\{1,a,a^2,a^5\right\}$.

All we have to do is to exclude the existence of a $5$-sided
and a $6$-sided Poncelet polygon for the following 16 oval pairs
$$ (O_t(b),O_s(f)),\ b \in \left\{1,a,a^2,a^5\right\},\ f \in \left\{0,a,a^2,a^3\right\}. $$
By  direct inspection,  we  count the  number  of   points  on
$O_s(f)$ that  are incident  with a tangent  of $O_t(b)$.  Table \ref{table}
contains these numbers.

\begin{table}[h!]
\caption{Number of points on $O_s(f)$ which are incident with a tangent of $O_t(b)$}
\label{table}
  \begin{center}
    \begin{tabular}{c| c|c|c|c|}
         & $O_t(1)$ & $O_t(a)$ & $O_t(a^2)$ & $O_t(a^5)$ \\
         \hline
       $O_s(0)$  & 8 & 6 & 4& 4\\
       \hline
       $O_s(a)$ & 4 & 6 & 4 & 8\\
       \hline
       $O_s(a^2)$ & 8 & 4 & 5 & 5 \\
       \hline
       $O_s(a^3)$ &4 &6 & 8& 4 \\
       \hline
    \end{tabular}
  \end{center}
\end{table}
Since by  construction all of  these pairs already form  one $4$-sided
Poncelet  polygon,  the condition  of  $9$  or  $10$ exterior  points  of
$O_t(b)$ on  $O_s(f)$ is  necessary to find  a $5$-sided  or $6$-sided
Poncelet polygon,  respectively. Since there  are at most  eight exterior
points of $O_t(b)$  on $O_s(f)$, we can exclude  their existence. This
completes the proof of $PG(2,9)$ being a Poncelet plane.\hfill$\Box$

\section{Poncelet's Theorem in the finite projective planes over $\mathfrak{S}$}
\subsection{The miniquaternion near-field $\mathfrak{S}$}
We describe the near-field we use to construct the
three non-Desarguesian finite projective  planes, denoted by $\Omega$,
$\Omega^D$  and $\Psi$.  All notations  and well-known  properties are
based on \cite{MR0298538}.

A \emph{finite near-field} is a system $(\mathfrak{S},+,\odot)$, such that
\begin{enumerate}\itemsep=0mm
  \item[(i)] $\mathfrak{S}$ is finite,
  \item[(ii)] $(\mathfrak{S},+)$ is a commutative group with identity $0$,
  \item[(ii)] the multiplication is a group operation on $\mathfrak{S}\backslash \left\lbrace 0\right\rbrace $ with identity $1$ and
  \item[(iv)]  the multiplication  is  right distributive  over the  addition, i.e. 
    $$(m+n)l=ml+nl, \forall \ m,n,l \in \mathfrak{S}.$$
\end{enumerate}
Note  that  we  do  not  necessarily need  the  multiplication  to  be
commutative, hence the left distribution law does not have to be valid
for all elements in the near-field.  This is exactly the property used
in the construction of the non-isomorphic planes of order 9. We need to
describe  addition  and  multiplication  for a  near-field  with  nine
elements. For this, consider
$$ \mathfrak{S} = \left\lbrace 0,1,-1,i,-i,j,-j,k,-k\right\rbrace $$
where we define
$$ j:=1+i,\ k:=1-i.$$
We can view the nine elements as elements over the basis $\left\lbrace
  1,i  \right\rbrace  $  and call  $\mathfrak{D}:=\left\lbrace  0,1,-1
\right\rbrace$ the  real elements  and $\mathfrak{S}^{*}:=\left\lbrace
  i,-i,j,-j,k,-k   \right\rbrace$  the   complex  elements.    By  the
definition of $j$ and $k$ above and taking the coefficients of $1$ and
$i$ modulo $3$, we are able to  add any two elements in the near-field
(Table \ref{AddTable}).  For the multiplication in  $\mathfrak{S}$, we
want to end up with a non-commutative operation. We use the relations
$$ i^2 = j^2 = k^2 =-1, \ ij=k=-ji,\ jk=i=-kj,\ ki=j=-ik $$
which again enables us to  multiply any two elements in $\mathfrak{S}$
(Table \ref{AddTable}).

\begin{table}[h!]
 \caption{Addition and multiplication of elements in $\mathfrak{S}$}
 \label{AddTable}
  \begin{center}
    \begin{tabular}{r|rrrrrrrrr }
       $+  $&$ 0 $&$ 1 $&$ -1 $&$ i $&$ -i $&$ j $&$ -j $&$ k $&$ -k$ \\
       \hline
       $0  $&$ 0 $&$ 1 $&$ -1 $&$ i $&$ -i $&$ j $&$ -j $&$ k $&$ -k$ \\
       $1 $&$ 1 $&$ -1 $&$ 0 $&$ j $&$ k $&$ -k $&$ -i $&$ -j $&$ i$   \\
       $-1 $&$ -1 $&$ 0 $&$ 1 $&$ -k $&$ -j $&$ i $&$ k $&$ -i $&$ j$   \\
       $i $&$ i $&$ j $&$ -k $&$ -i $&$ 0 $&$ k $&$ -1 $&$ 1 $&$ -j$   \\
       $-i $&$ -i $&$ k $&$ -j $&$ 0 $&$ i $&$ 1 $&$ -k $&$ j $&$ -1$   \\
       $j $&$ j $&$ -k $&$ i $&$ k $&$ 1 $&$ -j $&$ 0 $&$ -1 $&$ -i$   \\
       $-j $&$ -j $&$ -i $&$ k $&$ -1 $&$ -k $&$ 0 $&$ j $&$ i $&$ 1$   \\
       $k $&$ k $&$ -j $&$ -i $&$ 1 $&$ j $&$ -1 $&$ i $&$ -k $&$ 0 $  \\
       $-k $&$ -k $&$ i $&$ j $&$ -j $&$ -1 $&$ -i $&$ 1 $&$ 0 $&$ k$  \\ 
      \\

      $\odot$  &$  0 $&$ 1 $&$ -1 $&$ i $&$ -i $&$ j $&$ -j $&$ k $&$ -k$  \\
       \hline
      $  0 $&$ 0 $&$ 0 $&$ 0  $&$ 0 $&$  0 $&$ 0 $&$  0 $&$ 0 $&$  0$  \\
      $  1 $&$ 0 $&$ 1 $&$ -1 $&$ i $&$ -i $&$ j $&$ -j $&$ k $&$ -k$    \\
       $ -1 $&$0 $&$  -1 $&$ 1 $&$ -i $&$ i $&$ -j $&$ j $&$ -k $&$ k$   \\
      $  i $&$ 0 $&$ i $&$ -i $&$ -1 $&$ 1 $&$ k $&$ -k $&$ -j $&$ j $   \\
      $  -i $&$0 $&$  -i $&$ i $&$ 1 $&$ -1 $&$ -k $&$k $&$ j $&$ -j$    \\
      $  j $&$0 $&$  j $&$ -j $&$ -k $&$ k $&$ -1 $&$ 1 $&$ i $&$ -i$    \\
      $  -j $&$0 $&$  -j $&$ j $&$ k $&$ -k $&$ 1 $&$ -1 $&$ -i $&$ i$   \\
      $ k $&$ 0 $&$ k $&$ -k $&$ j $&$ -j $&$ -i $&$ i $&$ -1 $&$ 1 $   \\
      $ -k $&$0 $&$  -k $&$ k $&$ -j $&$ j $&$ i $&$ -i $&$ 1 $&$ -1$   
    \end{tabular}
  \end{center}
\end{table}

Note  that this  is the  multiplication law  of the  quaternion group,
which explains  the name `miniquaternion near-field'.  Note also, that
in this near-field the left distribution law does not hold in general,
as for example $ i (j+k) = i (-1) = -i$ but $ij+ik=k-j=i$.
\subsection{The plane $\Omega$}
In  order to  construct the finite  projective plane  $\Omega$ of  order 9
using the near-field $\mathfrak{S}$ we  start with an affine plane and
extend  it to  a projective  plane. We  distinguish
between so-called \emph{proper points} on  $\Omega$, which are in affine form
$(x,y)$, and  \emph{ideal points},  which connect  the parallel  lines.  More
precisely, the points of $\Omega$ are given by\vspace*{-2mm}
\begin{itemize}\itemsep=0mm
  \item[-] 81 proper points of the form $(x,y)$, for $x,y \in \mathfrak{S}$,
  \item[-] 9 ideal points of the form $(1,y,0)$ for $y \in \mathfrak{S}$ and
  \item[-] one ideal point of the form $(0,1,0)$.
\end{itemize}\vspace*{-2mm}
The lines of $\Omega$ are given by\vspace*{-2mm}
\begin{itemize}\itemsep=0mm
  \item[-] 81 proper lines of the form $y=x \mu + \nu $ for $\mu, \nu \in \mathfrak{S}$, denoted by $(\mu,\nu)$,
  \item[-] 9 proper lines of the form $x=\lambda$ for $\lambda \in \mathfrak{S}$, denoted by $(\lambda)$, and
  \item[-] one ideal line, denoted by $\mathfrak{I}$.
\end{itemize}
On the line $y=x \mu + \nu $ there are nine proper points $(x,y)$ and the ideal point $(1,\mu,0)$.
On the line $x=\lambda$ there are nine proper points $(\lambda,y)$ and the ideal point $(0,1,0)$.
All 10 ideal points are on the ideal line $\mathfrak{I}$.

It is crucial  to consider $y = x \mu  + \nu$ instead of $y =  \mu x +
\nu$, since  multiplication is not  commutative. It can be  shown that
the above  defined points and  lines with the incidence  relation give
indeed a  finite projective  plane of order  9 (see \cite{MR0298538}).

Now we want to find ovals in  the plane $\Omega$, i.e.\ we want to find
sets of  $10$ points, no  three of  which are collinear.   Compared to
finite projective coordinate  planes, it is much harder  to find ovals
in  this  plane,   since  ovals  cannot  be   described  by  quadratic
forms. Hence, for  a set of $10$  points, we have to  search all lines
connecting them  to be sure that  no three of them  are collinear. For
any point on the oval, we end up with nine lines connecting this point
to the  other points  on the oval,  and all these  secants have  to be
different.  Similarly it is much harder to find the tangent in a point
of the oval.  Nevertheless, the set $O_1$ given by
$$ \left\lbrace \begin{pmatrix} -1\\i \end{pmatrix},\begin{pmatrix} 0\\1 \end{pmatrix},\begin{pmatrix} 1\\0 \end{pmatrix},
\begin{pmatrix} -i\\-j \end{pmatrix},\begin{pmatrix} i\\j \end{pmatrix},\begin{pmatrix} -j\\-1 \end{pmatrix},
\begin{pmatrix} j\\k \end{pmatrix},\begin{pmatrix} -k\\-i \end{pmatrix},\begin{pmatrix} 0\\1\\0 \end{pmatrix},\begin{pmatrix} 1\\0\\0 \end{pmatrix}\right\rbrace $$ 
is an example of an oval in $\Omega$.
To see  this, we  have to
calculate all  secants and check,  whether they are  different.  Table
\ref{Oval1} shows all secants and tangents of $ O_1$.

\begin{table}[h]
  \caption{The diagonal entries are the tangents of $O_1$ and the 
    other entries are the secants of  the points of $O_1$ 
    listed in the first row and column.}
    \small 
\label{Oval1}
  \begin{center}{\setlength{\tabcolsep}{0.6mm}
    \begin{tabular}{c|c|c|c|c|c|c|c|c|c|c}
        $O_1$ &$\begin{pmatrix} -1\\i \end{pmatrix}$&$\begin{pmatrix} 0\\1 \end{pmatrix}$&$\begin{pmatrix} 1\\0 \end{pmatrix}$&$\begin{pmatrix} -i\\-j \end{pmatrix}$&$\begin{pmatrix} i\\j \end{pmatrix}$&$\begin{pmatrix} -j\\-1 \end{pmatrix}$&$\begin{pmatrix} j\\k \end{pmatrix}$&$\begin{pmatrix} -k\\-i \end{pmatrix}$&$\begin{pmatrix} 0\\1\\0 \end{pmatrix}$&$\begin{pmatrix} 1\\0\\0 \end{pmatrix}$\\
       \hline
        $\begin{pmatrix} -1\\i \end{pmatrix}$&$\begin{pmatrix} j\\k \end{pmatrix}$&$\begin{pmatrix} k\\1 \end{pmatrix}$&$\begin{pmatrix} i\\-i \end{pmatrix}$&$\begin{pmatrix} -1\\-k \end{pmatrix}$&$\begin{pmatrix} -j\\-1 \end{pmatrix}$&$\begin{pmatrix} -k\\-j \end{pmatrix} $&$ \begin{pmatrix} -i\\0 \end{pmatrix} $&$\begin{pmatrix} 1\\j \end{pmatrix}$&$(-1)$&$\begin{pmatrix} 0\\i \end{pmatrix}$\\
        \hline
        $\begin{pmatrix} 0\\1 \end{pmatrix} $ && $\begin{pmatrix} -i\\1 \end{pmatrix}$&$\begin{pmatrix} -1\\1 \end{pmatrix}$&$\begin{pmatrix} -j\\1 \end{pmatrix}$&$\begin{pmatrix} 1\\1 \end{pmatrix}$&$\begin{pmatrix} j\\1 \end{pmatrix}$&$\begin{pmatrix} -k\\1 \end{pmatrix}$&$\begin{pmatrix} i\\1 \end{pmatrix}$&$(0)$&$\begin{pmatrix} 0\\1 \end{pmatrix}$\\
        \hline
        $\begin{pmatrix} 1\\0 \end{pmatrix} $ &&& $\begin{pmatrix} -j\\j \end{pmatrix} $&$\begin{pmatrix} 1\\-1 \end{pmatrix} $&$ \begin{pmatrix} -i\\i \end{pmatrix} $&$ \begin{pmatrix} k\\-k \end{pmatrix}$&$ \begin{pmatrix} j\\-j \end{pmatrix} $&$ \begin{pmatrix} -k\\k \end{pmatrix} $&$ (1) $&$\begin{pmatrix} 0\\0 \end{pmatrix}$\\
        \hline
        $\begin{pmatrix} -i\\-j \end{pmatrix}$ &&&& $ \begin{pmatrix} i\\k \end{pmatrix} $&$\begin{pmatrix} -k\\0 \end{pmatrix}$&$\begin{pmatrix} -i\\-i \end{pmatrix}$&$ \begin{pmatrix} k\\j\end{pmatrix} $&$\begin{pmatrix} j\\i \end{pmatrix}$&$ (-i) $&$\begin{pmatrix} 0\\-j \end{pmatrix}$\\
        \hline
        $\begin{pmatrix} i\\j \end{pmatrix}$ &&&&& $\begin{pmatrix} j\\-i \end{pmatrix} $&$ \begin{pmatrix} -1\\k \end{pmatrix}$&$ \begin{pmatrix} i\\-k \end{pmatrix}$&$ \begin{pmatrix} k\\-j \end{pmatrix}$&$ (i) $&$\begin{pmatrix} 0\\j \end{pmatrix}$ \\
        \hline
        $\begin{pmatrix} -j\\-1 \end{pmatrix}$ &&&&&& $ \begin{pmatrix} i\\j \end{pmatrix} $&$ \begin{pmatrix} 1\\i \end{pmatrix} $&$ \begin{pmatrix} -j\\0 \end{pmatrix} $&$ (-j) $&$ \begin{pmatrix} 0\\-1 \end{pmatrix}$\\
        \hline
        $\begin{pmatrix} j\\k \end{pmatrix}$ &&&&&&& $ \begin{pmatrix} -j\\-i \end{pmatrix} $&$\begin{pmatrix} -1\\-1 \end{pmatrix} $&$ (j) $&$\begin{pmatrix} 0\\k \end{pmatrix}$\\
        \hline
        $\begin{pmatrix} -k\\-i \end{pmatrix}$ &&&&&&&& $ \begin{pmatrix} -i\\-k \end{pmatrix} $&$ (-k) $&$ \begin{pmatrix} 0\\-i \end{pmatrix}$ \\
        \hline
        $\begin{pmatrix} 0\\1\\0 \end{pmatrix}$ &&&&&&&&& $ (k) $& $\mathfrak{I}$ \\
        \hline
        $\begin{pmatrix} 1\\0\\0 \end{pmatrix}$ &&&&&&&&&& $\begin{pmatrix} 0\\-k \end{pmatrix} $
    \end{tabular}}
  \end{center}
\end{table}

Recall that in $PG(2,9)$, Pascal's Theorem plays a central r\^ole
in the proof of Poncelet's Closure  Theorem. We will see that Pascal's
Theorem is not  true in general in the plane  $\Omega$. For this, take
for example the six points
$$ A=(-1,i),\ B=(0,1),\ C=(1,0),\ D=(i,j),\ E=(-k,-i),\ F=(0,1,0). $$
These  points all lie on $O_1$, hence they lie indeed on a non-degenerate hexagon. We have
\begin{align*} 
\overline{AB}&: y=xk+1,& \overline{BC}&: y=-x+1,& \overline{CD}&:y=-xi+1 \\
\overline{DE}&: y=xk-j,& \overline{EF}&: x=-k,& \overline{FA}&: x=-1.
\end{align*}
The intersection points we need in Pascal's Theorem are given by
$$ P=(1,k,0),\ Q=(-k,-j),\ R=(-1,-i). $$
These are  \emph{not} collinear, as  the line through $P$  and $Q$ is
given by $y=xk+k$ and the line through $P$ and $R$ is $y=xk+j$.
\begin{theorem} The finite projective plane $\Omega$ of order 9 is not a Poncelet plane.
\end{theorem}
\begin{proof} We have to find a  pair of ovals $(O_t,O_s$) which carries
  at the same time an $n$-sided and an $m$-sided Poncelet polygon with
  $m\neq n$ and $m,n \geq 3$.  For $O_t$ we  take $O_1$, and  for $O_s$ we  choose the oval
$$ \left\lbrace \begin{pmatrix} 0\\j \end{pmatrix},\begin{pmatrix} i\\i \end{pmatrix},\begin{pmatrix} i\\-k \end{pmatrix},
\begin{pmatrix} -j\\j \end{pmatrix},\begin{pmatrix} j\\0 \end{pmatrix},\begin{pmatrix} j\\-j \end{pmatrix},
\begin{pmatrix} -k\\i \end{pmatrix},\begin{pmatrix} -k\\-k \end{pmatrix},\begin{pmatrix} 1\\-j\\0 \end{pmatrix},\begin{pmatrix} 1\\j\\0 \end{pmatrix}\right\rbrace $$ 
Again, we have to ensure that all secants of this set are different. For
this, see Table \ref{Oval2}.

\begin{table}[h!]
\caption{The diagonal entries indicate the tangents of  $O_s$, 
the other entries indicate the secants.}
\small  
\label{Oval2}
  \begin{center}{\setlength{\tabcolsep}{0.6mm}
    \begin{tabular}{c|c|c|c|c|c|c|c|c|c|c}
     $O_s$&$\begin{pmatrix} 0\\j \end{pmatrix}$&$\begin{pmatrix} i\\i \end{pmatrix}$&$\begin{pmatrix} i\\-k \end{pmatrix}$&$\begin{pmatrix} -j\\j \end{pmatrix}$&$\begin{pmatrix} j\\0 \end{pmatrix}$&$\begin{pmatrix} j\\-j \end{pmatrix}$&$\begin{pmatrix} -k\\i \end{pmatrix}$&$\begin{pmatrix} -k\\-k \end{pmatrix}$&$\begin{pmatrix} 1\\-j\\0 \end{pmatrix}$&$\begin{pmatrix} 1\\j\\0 \end{pmatrix}$\\
       \hline
     $\begin{pmatrix} 0\\j \end{pmatrix}$&$(0)$&$\begin{pmatrix} i\\j \end{pmatrix}$&$\begin{pmatrix} -i\\j \end{pmatrix}$&$\begin{pmatrix} 0\\j \end{pmatrix}$&$\begin{pmatrix} -1\\j \end{pmatrix}$&$\begin{pmatrix} 1\\j \end{pmatrix}$&$ \begin{pmatrix} -k\\j \end{pmatrix}$&$\begin{pmatrix} k\\j \end{pmatrix}$&$ \begin{pmatrix} -j\\j \end{pmatrix}$&$ \begin{pmatrix} j\\j \end{pmatrix}$ \\
     \hline
     $\begin{pmatrix} i\\i \end{pmatrix}$& &$\begin{pmatrix} -1\\-i \end{pmatrix}$&$(i)$&$\begin{pmatrix} k\\k \end{pmatrix}$&$\begin{pmatrix} -i\\-k \end{pmatrix} $&$ \begin{pmatrix} -k\\-1 \end{pmatrix} $&$\begin{pmatrix} 0\\i \end{pmatrix}$&$\begin{pmatrix} 1\\0 \end{pmatrix} $&$ \begin{pmatrix} -j\\1 \end{pmatrix} $&$ \begin{pmatrix} j\\-j \end{pmatrix}$ \\
     \hline
     $\begin{pmatrix} i\\-k \end{pmatrix}$ && $ $&$\begin{pmatrix} 1\\-1 \end{pmatrix} $&$ \begin{pmatrix} -k\\1 \end{pmatrix} $&$ \begin{pmatrix} k\\-i \end{pmatrix} $&$ \begin{pmatrix} i\\i \end{pmatrix} $&$\begin{pmatrix} -1\\-j \end{pmatrix}$&$\begin{pmatrix} 0\\-k \end{pmatrix}$&$\begin{pmatrix} -j\\0 \end{pmatrix}$&$\begin{pmatrix} j\\k \end{pmatrix}$ \\
     \hline
     $\begin{pmatrix} -j\\j \end{pmatrix}$ &&&& $(-j)$&$\begin{pmatrix} 1\\-j \end{pmatrix}$&$\begin{pmatrix} -1\\0 \end{pmatrix}$&$\begin{pmatrix} -i\\-1 \end{pmatrix}$&$\begin{pmatrix} i\\-i \end{pmatrix}$
     &$\begin{pmatrix} -j\\-k \end{pmatrix}$&$\begin{pmatrix} j\\i \end{pmatrix}$\\
     \hline
     $\begin{pmatrix} j\\0 \end{pmatrix}$ &&&&& $\begin{pmatrix} 0\\0 \end{pmatrix}$&$(j)$&$\begin{pmatrix} i\\k \end{pmatrix}$&$\begin{pmatrix} -k\\i \end{pmatrix}$&$\begin{pmatrix} -j\\-i \end{pmatrix}$&
     $\begin{pmatrix} j\\1 \end{pmatrix}$\\
     \hline
     $\begin{pmatrix} j\\-j \end{pmatrix}$&&&&&&$\begin{pmatrix} 0\\-j \end{pmatrix}$&$\begin{pmatrix} k\\-k \end{pmatrix}$&$\begin{pmatrix} -i\\1 \end{pmatrix}$&$\begin{pmatrix} -k\\k \end{pmatrix}$&
     $\begin{pmatrix} j\\-i \end{pmatrix}$\\
     \hline
     $\begin{pmatrix} -k\\i \end{pmatrix}$ &&&&&&& $\begin{pmatrix} 1\\1 \end{pmatrix}$&$(-k)$&$\begin{pmatrix} -j\\-i \end{pmatrix}$&$\begin{pmatrix} j\\0 \end{pmatrix}$\\
     \hline
     $\begin{pmatrix} -k\\-k \end{pmatrix}$&&&&&&&&$\begin{pmatrix} -1\\k \end{pmatrix}$&$\begin{pmatrix} -j\\-j \end{pmatrix}$&$\begin{pmatrix} j\\-1 \end{pmatrix}$\\
     \hline
     $\begin{pmatrix} 1\\-j\\0 \end{pmatrix}$ &&&&&&&&& $\begin{pmatrix} -j\\i \end{pmatrix}$&$\mathfrak{I}$\\
     \hline
     $\begin{pmatrix} 1\\j\\0 \end{pmatrix}$ &&&&&&&&&& $\begin{pmatrix} j\\-k \end{pmatrix}$ 
    \end{tabular}}
  \end{center}
\end{table}

For  this oval  pair $(O_t,O_s)$,  we  can now  find simultaneously  a
$5$-sided Poncelet  polygon and a  $4$-sided Poncelet polygon.  To see
this, start with the point $(0,j)$  on $O_s$. The line joining $(0,j)$
and   $(1,-j,0)$  is   the   line  $y=-xj+j$,   as   given  in   Table
\ref{Oval2}. This line is a tangent  of oval $O_t$, namely the tangent
in the point $(1,0)$. Continuing with $(1,-j,0)$, we see that the line
joining $(1,-j,0)$ and $(-k,i)$ is  $y=-xj-i$, which is the tangent of
$O_t$ in  the point $(j,k)$.  Moreover, the line joining  $(-k,i)$ and
$(j,0)$ is $y=xi+k$,  which is the tangent of $O_t$  in $(-i,-j)$. The
line  joining $(j,0)$  and  $(i,i)$  is $y=-xi-k$,  which  is again  a
tangent of $O_t$, namely in the point $(-k,-i)$. For the last step, we
see that  the line joining $(i,i)$  and $(0,j)$ is $y=xi+j$,  which is
the tangent  of $O_t$ in the  point $(-j,-1)$. This gives  a $5$-sided
Poncelet polygon for  this pair. Similarly, by  starting with $(i,-k)$
on  $O_s$,  a  $4$-sided  Poncelet  polygon occurs.  To
summarize the result, we have the $5$-sided Poncelet polygon
\begin{align*}
\begin{pmatrix} 0\\j \end{pmatrix} 
&\xrightarrow{(1,0)}\begin{pmatrix} 1\\-j\\0 \end{pmatrix}\xrightarrow{(j,k)} \begin{pmatrix} -k\\i \end{pmatrix} \xrightarrow{(-i,-j)}\begin{pmatrix} j\\0 \end{pmatrix}
\xrightarrow{(-k,-i)} \begin{pmatrix} i\\i \end{pmatrix}\xrightarrow{(-j,-1)} \begin{pmatrix} 0\\j \end{pmatrix}
\end{align*}
and the $4$-sided Poncelet polygon
\begin{align*}
\begin{pmatrix} i\\-k \end{pmatrix}
\xrightarrow{(-1,i)} \begin{pmatrix} 1\\j\\0 \end{pmatrix} 
\xrightarrow{(i,j)} \begin{pmatrix} j\\-j \end{pmatrix}
\xrightarrow{(0,1)} \begin{pmatrix} -k\\-k \end{pmatrix} 
\xrightarrow{(1,0,0)} \begin{pmatrix} i\\-k \end{pmatrix}.
\end{align*}
The remaining  point $(-j,j)$  on $O_s$  is an  inner point  of $O_t$,
which means  that it is not  incident with any tangent  of $O_t$. This
pair $(O_t,O_s)$  is therefore no  Poncelet $m$-pair for  any possible
value of $m$, which shows that $\Omega$ is not a Poncelet plane.
\end{proof}
This pair of ovals  gives even one more proof of  $\Omega$ not being a
Poncelet plane, namely by changing the  r\^oles of $O_t$ and $O_s$. If
we consider points on $O_t$ and tangents of $O_s$,  
we find  simultaneously a  $4$-sided and  a $3$-sided  Poncelet
polygon, namely
$$ 
\begin{pmatrix} 0\\1 \end{pmatrix} \xrightarrow{(-k,i)} \begin{pmatrix} i\\j \end{pmatrix} 
\xrightarrow{(-k,-k)} \begin{pmatrix} -j\\-1 \end{pmatrix} 
\xrightarrow{(-j,j)} \begin{pmatrix} 0\\1\\0 \end{pmatrix}
\xrightarrow{(0,j)} \begin{pmatrix} 0\\1 \end{pmatrix} 
$$
and
$$ \begin{pmatrix} 1\\0 \end{pmatrix} 
\xrightarrow{(i,-k)} \begin{pmatrix} -i\\-j \end{pmatrix} 
\xrightarrow{(j,-j)} \begin{pmatrix} 1\\0\\0 \end{pmatrix}
\xrightarrow{(j,0)} \begin{pmatrix} 1\\0 \end{pmatrix}.
$$
This shows that in $\Omega$,  Poncelet's Theorem is not even partially
true for  only $3$-sided  polygons or only  $4$-sided polygons,  as we
could find counter examples in both cases.

\subsection{The plane $\Omega^D$}

Now, we want to  look at the dual plane of $\Omega$,  that is, we want
to  change  the  r\^ole  of  the  points  and  lines  constructed  for
$\Omega$ to obtain the plane $\Omega^D$. 
Note that $\Omega^D$  is indeed not isomorphic to $\Omega$
(see \cite{MR0298538}).

Recall the incidence relation for $\Omega$, given by
$$ (x,y) \in (\mu,\nu) \Leftrightarrow y = x \mu + \nu. $$
By changing the r\^oles of points and lines, $(x,y)$ denotes a line in
$\Omega^D$ and $(\mu,\nu)$  denotes a point in  $\Omega^D$. Hence, the
incidence relation becomes
$$ (\mu,\nu) \in (x,y) \Leftrightarrow \nu=-x \mu +y. $$
For   the  line   $x=\lambda$,  the   incidence  relation   stays  the
same. Moreover, the ideal line is not changed either. By adjusting the
notation  above  by taking  $x$  instead  of  $-x$,  we obtain  the
incidence relations for $\Omega^D$. Namely,
on the line $y= \mu x + \nu $, there are nine proper points $(x,y)$ and the ideal point $(1,\mu,0)$.
On the line $x=\lambda$, there are nine proper points $(\lambda,y)$ and the ideal point $(0,1,0)$.
All 10 ideal points are on the ideal line $\mathfrak{I}$.

\begin{theorem} The finite projective plane $\Omega^D$ of order 9 is not a Poncelet plane.
\end{theorem}

\begin{proof}
In  order to prove  that $\Omega^D$ is  not a
Poncelet  plane either  we can  dualize the  ovals from  the previous
section. For this, recall the oval $O_t$ given by 
$$ \left\lbrace \begin{pmatrix} -1\\i \end{pmatrix},\begin{pmatrix} 0\\1 \end{pmatrix},\begin{pmatrix} 1\\0 \end{pmatrix},\begin{pmatrix} -i\\-j \end{pmatrix},
\begin{pmatrix} i\\j \end{pmatrix},\begin{pmatrix} -j\\-1 \end{pmatrix},\begin{pmatrix} j\\k \end{pmatrix},\begin{pmatrix} -k\\-i \end{pmatrix},
\begin{pmatrix} 0\\1\\0 \end{pmatrix},\begin{pmatrix} 1\\0\\0 \end{pmatrix} \right\rbrace $$
in $\Omega$. The set of tangents of this oval is given by the lines
$$ \left\lbrace \begin{pmatrix} j\\k \end{pmatrix},\begin{pmatrix} -i\\1 \end{pmatrix},\begin{pmatrix} -j\\j \end{pmatrix},\begin{pmatrix} i\\k \end{pmatrix},\begin{pmatrix} j\\-i \end{pmatrix},
\begin{pmatrix} i\\j \end{pmatrix},\begin{pmatrix} -j\\-i \end{pmatrix},\begin{pmatrix} -i\\-k \end{pmatrix},(k),\begin{pmatrix} 0\\-k \end{pmatrix}   \right\rbrace. $$
Note that for the proper lines  $(\mu,\nu)$, we have to take the minus
sign  for  the   $x$-coordinate.  This  gives  the oval $O_t^D$
$$\left\lbrace \begin{pmatrix} -j\\k \end{pmatrix},\begin{pmatrix} i\\1 \end{pmatrix},\begin{pmatrix} j\\j \end{pmatrix},\begin{pmatrix} -i\\k \end{pmatrix},\begin{pmatrix} -j\\-i \end{pmatrix},
\begin{pmatrix} -i\\j \end{pmatrix},\begin{pmatrix} j\\-i \end{pmatrix},\begin{pmatrix} i\\-k \end{pmatrix},(k),\begin{pmatrix} 0\\-k \end{pmatrix}  \right\rbrace $$ 
in $\Omega^D$.

Similarly, the dualization of the oval $O_s$ leads to another oval in $\Omega^D$, namely $O_s^D$ given by
\small
$$ \left\lbrace \begin{pmatrix} 1\\0\\0 \end{pmatrix},\begin{pmatrix} 1\\-i \end{pmatrix},\begin{pmatrix} -1\\-1 \end{pmatrix},\begin{pmatrix} 1\\-j\\0 \end{pmatrix},\begin{pmatrix} 0\\0 \end{pmatrix},
\begin{pmatrix} 0\\-j \end{pmatrix},\begin{pmatrix} -1\\1 \end{pmatrix},\begin{pmatrix} 1\\k \end{pmatrix},\begin{pmatrix} j\\i\end{pmatrix},\begin{pmatrix} -j\\-k \end{pmatrix}\right\rbrace. $$
\normalsize
Now we can dualize the $n$-sided Poncelet polygons as well. 
Recall that for $O_t$ and $O_s$ in $\Omega$, we had the $4$-sided Poncelet polygon 
$$ \begin{pmatrix} i\\-k \end{pmatrix} 
\xrightarrow{(-1,i)} \begin{pmatrix} 1\\j\\0 \end{pmatrix} 
\xrightarrow{(i,j)} \begin{pmatrix} j\\-j \end{pmatrix} 
\xrightarrow{(0,1)} \begin{pmatrix} -k\\-k \end{pmatrix} 
\xrightarrow{(1,0,0)} \begin{pmatrix} i\\-k \end{pmatrix}
$$
for vertices on $O_s$.

The tangent of  $O_t$ in $(-1,i)$ is $(j,k)$, hence  we can start with
the corresponding  point $(-j,k)$  on $O_t^D$.  The connection  of two
points of $O_s$ in $\Omega$ is now the same as the intersection of two
tangents of $O_s^D$. Hence, in $\Omega^D$, we have to take vertices on
$O_t^D$  and tangents  of  $O_s^D$. The  next  tangent we  consider,
namely the  tangent in $(i,j)$,  corresponds to $(-j,-i)$, a  point on
$O_t^D$. The line connecting these two  points is the tangent of $O_s$
in  $(1,j,0)$  in  $\Omega$,  which  gives  $(-j,-k)$  interpreted  in
$\Omega^D$. When  performing this  operation with all  points obtained
before, we get the  $4$-sided Poncelet polygon 
$$\begin{pmatrix} -j\\k \end{pmatrix}
\xrightarrow{(-j,-k)} \begin{pmatrix} -j\\-i \end{pmatrix} 
\xrightarrow{(0,-j)} \begin{pmatrix} i\\1 \end{pmatrix}
\xrightarrow{(1,k)} \begin{pmatrix} 0\\-k \end{pmatrix} 
\xrightarrow{(-1,-1)} \begin{pmatrix} -j\\k \end{pmatrix}
$$
and $5$-sided
Poncelet polygon
$$
\begin{pmatrix} j\\j \end{pmatrix}
\xrightarrow{(j,i)} \begin{pmatrix} j\\-i \end{pmatrix} 
\xrightarrow{(-1,1)} \begin{pmatrix} -i\\k \end{pmatrix}
\xrightarrow{(0,0)}\begin{pmatrix} i\\-k \end{pmatrix} 
\xrightarrow{(1,-i)} \begin{pmatrix} -i\\j \end{pmatrix} 
\xrightarrow{(1,0,0)} \begin{pmatrix} j\\j \end{pmatrix}.
$$
Hence, $\Omega^D$ is not a Poncelet plane either.
\end{proof}

\subsection{The plane $\Psi$}

Similarly to  the construction of  the plane $\Omega$, we  can define
the  plane $\Psi$  using again the  fact  that $\mathfrak{S}$  is not  left
distributive. We make use of the homogeneous approach, unlike
the affine approach before.
A  point   is  defined  as  the   set  of  vectors
  $\left\lbrace  P \kappa  , \kappa  \in \mathfrak{S},  \kappa \neq  0
  \right\rbrace  $,  $P\in  \mathfrak{S}^3\setminus  \{(0,0,0)\}$.  A
  point  is  called  \emph{real},  if  there exists  a  non-zero  $\kappa$  in
  $\mathfrak{S}$,  such that  all coordinates  of $P  \kappa $  are in
  $\mathfrak{D}$. Otherwise, the point is called \emph{complex}.
Note  that there  are  13  real points and 78  complex points.

The \emph{line} through $P$ and $Q$ is defined by
$$ \left\lbrace P \right\rbrace \cup \left\lbrace P \kappa + Q, \kappa \in \mathfrak{S} \right\rbrace. $$
A line  is called \emph{real} if  at least two  real points are on  the line,
otherwise \emph{complex}.
We can choose the line  at infinity $z=0$. All points
not on this line can be parameterized by $P=(x,y,1)$ and all points on
the line $z=0$ can be seen as $Q=(1,\kappa,0)$. 
We get 13 real lines, namely\vspace*{-2mm}
\begin{itemize}\itemsep=0mm
  \item[-] 9 lines of the form $y=mx+c,\ m,c \in \mathfrak{D}$, denoted by $(m,c,1)$,
  \item[-] 3 lines of the form $x=c,\ c\in \mathfrak{D}$, denoted by $(c,1,0)$, and
  \item[-] one line $z=0$, denoted by $(0,0,0)$.
\end{itemize}\vspace*{-2mm}
The 78 complex lines are given by\vspace*{-2mm}
\begin{itemize}\itemsep=0mm
\item[-]  54  lines of  the form  $y-s=\kappa  (x-r), r,s  \in
  \mathfrak{D},\    \kappa    \in    \mathfrak{S}^*$,    denoted    by
  $(s,r,\kappa)$,
\item[-] 18  lines   of  the  form  $y=mx+\kappa,\   m  \in
  \mathfrak{D},\    \kappa    \in    \mathfrak{S}^*$,    denoted    by
  $(m,\kappa,1)$, and
\item[-] 6  lines $x=\kappa,\  \kappa  \in  \mathfrak{S}^*$,
  denoted by $(\kappa,1,0)$.
 \end{itemize} 
 Note that we  have parameterized the lines and points  in a different
 way, since for example the  vectors $(1,1,i)$ and $(-1,-1,-i)$ do not
 represent the same line, but they do represent the same point. It can
 be  shown that  these points  and lines  together with  the incidence
 relations form indeed  a finite projective plane of order  9 which is
 not isomorphic  to $\Omega$  or $\Omega^D$,  and $\Psi$  is self-dual
 (see \cite{MR0298538}).
 
 \begin{theorem} The finite projective plane $\Psi$ of order 9 is not a Poncelet plane.
\end{theorem}
 
 \begin{proof}
 Look at the two ovals $O_t$
 \small
$$ \left\lbrace \begin{pmatrix} -1\\-1\\1 \end{pmatrix}, \begin{pmatrix} -1\\1\\1 \end{pmatrix}, \begin{pmatrix} 0\\-i\\1 \end{pmatrix},
           \begin{pmatrix} 0\\i\\1 \end{pmatrix},\begin{pmatrix} 1\\0\\1 \end{pmatrix}, \begin{pmatrix} 1\\0\\1 \end{pmatrix},
           \begin{pmatrix} j\\-k\\1 \end{pmatrix}, \begin{pmatrix} j\\k\\1 \end{pmatrix},\begin{pmatrix} k\\-j\\1 \end{pmatrix},
           \begin{pmatrix} k\\j\\1 \end{pmatrix} \right\rbrace $$
\normalsize
and $O_s$
\small
$$ \left\lbrace\begin{pmatrix} -1\\i\\1 \end{pmatrix}, \begin{pmatrix} -1\\k\\1 \end{pmatrix}, \begin{pmatrix} 0\\-1\\1 \end{pmatrix}, \begin{pmatrix}0\\0\\1 \end{pmatrix}, 
\begin{pmatrix} 1\\-1\\1 \end{pmatrix}, \begin{pmatrix} 1\\1\\1 \end{pmatrix}, \begin{pmatrix} 1\\i\\1 \end{pmatrix}, \begin{pmatrix} 1\\k\\1 \end{pmatrix}, 
\begin{pmatrix} i\\k\\1 \end{pmatrix}, \begin{pmatrix} k\\i\\1 \end{pmatrix}\right\rbrace.$$
\normalsize
Similarly to  the approach before, in  Table \ref{tab5} and \ref{tab6}
we just  list all secants  and tangents to  check that these  sets are
indeed ovals.
\begin{table}
\caption{Oval $O_t$ in $\Psi$}\label{tab5}
  \footnotesize
  \begin{center}{\setlength{\tabcolsep}{0.5mm}
    \begin{tabular}{c| c|c|c|c|c|c|c|c|c|c}
           $O_t$  &$\begin{pmatrix} -1\\-1\\1 \end{pmatrix}$&$ \begin{pmatrix} -1\\1\\1 \end{pmatrix}$&$ \begin{pmatrix} 0\\-i\\1 \end{pmatrix}$&
           $ \begin{pmatrix} 0\\i\\1 \end{pmatrix}$&$ \begin{pmatrix} 1\\0\\1 \end{pmatrix}$&$ \begin{pmatrix} 1\\0\\1 \end{pmatrix}$&
           $ \begin{pmatrix} j\\-k\\1 \end{pmatrix}$&$ \begin{pmatrix} j\\k\\1 \end{pmatrix}$&$\begin{pmatrix} k\\-j\\1 \end{pmatrix}$&
           $ \begin{pmatrix} k\\j\\1 \end{pmatrix}$ \\
       \hline
        $\begin{pmatrix} -1\\-1\\1 \end{pmatrix}   $&$ \begin{pmatrix} 1\\0\\1 \end{pmatrix}$&$ \begin{pmatrix} -1\\1\\0 \end{pmatrix}$&$\begin{pmatrix} -1\\-1\\k \end{pmatrix}$&
        $ \begin{pmatrix} -1\\-1\\j \end{pmatrix}$&$ \begin{pmatrix} 0\\-1\\1 \end{pmatrix}$&$\begin{pmatrix} -1\\1\\1 \end{pmatrix}$&$ \begin{pmatrix} -1\\-1\\-j \end{pmatrix}$&
        $ \begin{pmatrix} -1\\-1\\-i \end{pmatrix}$&$ \begin{pmatrix} -1\\-1\\-k \end{pmatrix}$&$ \begin{pmatrix} -1\\-1\\i \end{pmatrix}$ \\
        \hline
        $\begin{pmatrix} -1\\1\\1 \end{pmatrix} $ && $ \begin{pmatrix} -1\\0\\1 \end{pmatrix} $&$ \begin{pmatrix} 1\\-1\\-j \end{pmatrix} $&$ \begin{pmatrix} 1\\-1\\-k \end{pmatrix} $&
        $ \begin{pmatrix} 0\\1\\1 \end{pmatrix} $&$ \begin{pmatrix} 1\\-1\\1 \end{pmatrix} $&$\begin{pmatrix} 1\\-1\\i \end{pmatrix} $&$ \begin{pmatrix} 1\\-1\\j \end{pmatrix} $&
        $ \begin{pmatrix} 1\\-1\\-i \end{pmatrix} $&$ \begin{pmatrix} 1\\-1\\k \end{pmatrix}$ \\
        \hline
        $\begin{pmatrix} 0\\-i\\1 \end{pmatrix} $ &&& $\begin{pmatrix} 0\\-1\\-i \end{pmatrix} $&$ \begin{pmatrix} 0\\1\\0 \end{pmatrix} $&$ \begin{pmatrix} 0\\-i\\1 \end{pmatrix}$&
        $ \begin{pmatrix} 0\\1\\i \end{pmatrix} $&$ \begin{pmatrix} -1\\-i\\1 \end{pmatrix}$&$ \begin{pmatrix} -1\\1\\-k \end{pmatrix}$&$ \begin{pmatrix} 1\\1\\j \end{pmatrix} $&
        $ \begin{pmatrix} 1\\-i\\1 \end{pmatrix}$ \\
        \hline
        $\begin{pmatrix} 0\\i\\1 \end{pmatrix}$ &&&& $ \begin{pmatrix} 0\\-1\\i \end{pmatrix} $&$ \begin{pmatrix} 0\\i\\1 \end{pmatrix}$&$ \begin{pmatrix} 0\\1\\-i \end{pmatrix} $&
        $\begin{pmatrix} 1\\1\\k \end{pmatrix} $&$ \begin{pmatrix} 1\\i\\1 \end{pmatrix} $&$\begin{pmatrix} -1\\i\\1 \end{pmatrix} $&$\begin{pmatrix} -1\\1\\-j \end{pmatrix}$\\
        \hline
        $\begin{pmatrix} 1\\0\\0 \end{pmatrix} $ &&&&& $ \begin{pmatrix} 0\\0\\0 \end{pmatrix} $&$ \begin{pmatrix} 0\\0\\1 \end{pmatrix} $&$ \begin{pmatrix} 0\\-k\\1 \end{pmatrix}$&
        $ \begin{pmatrix} 0\\k\\1 \end{pmatrix} $&$ \begin{pmatrix} 0\\-j\\1 \end{pmatrix} $&$ \begin{pmatrix} 0\\j\\1 \end{pmatrix}$ \\
        \hline
        $\begin{pmatrix} 1\\0\\1 \end{pmatrix}$ &&&&&& $ \begin{pmatrix} 1\\1\\0 \end{pmatrix}$&$ \begin{pmatrix} 0\\1\\j \end{pmatrix} $&$ \begin{pmatrix} 0\\1\\-j \end{pmatrix}$&
        $ \begin{pmatrix} 0\\1\\k \end{pmatrix}$&$ \begin{pmatrix} 0\\1\\-k \end{pmatrix}$\\
        \hline
        $ \begin{pmatrix} j\\-k\\1 \end{pmatrix}$ &&&&&&& $ \begin{pmatrix} -1\\0\\-k \end{pmatrix} $&$\begin{pmatrix} j\\1\\0 \end{pmatrix} $&$ \begin{pmatrix} 1\\1\\1 \end{pmatrix} $&
        $ \begin{pmatrix} 0\\0\\-i\end{pmatrix}$\\
        \hline
        $\begin{pmatrix} j\\k\\1 \end{pmatrix} $ &&&&&&&& $ \begin{pmatrix} 1\\0\\k \end{pmatrix} $&$ \begin{pmatrix} 0\\0\\i \end{pmatrix} $&$ \begin{pmatrix} -1\\-1\\1 \end{pmatrix}$ \\
        \hline
        $\begin{pmatrix} k\\-j\\1 \end{pmatrix} $ &&&&&&&&& $ \begin{pmatrix} -1\\0\\-j \end{pmatrix} $&$ \begin{pmatrix} k\\1\\0 \end{pmatrix}$ \\
        \hline
        $\begin{pmatrix} k\\j\\1 \end{pmatrix}$ &&&&&&&&&& $ \begin{pmatrix} 1\\0\\j \end{pmatrix}$ 
    \end{tabular}}
  \end{center}
\end{table}

\begin{table}
\caption{Oval $O_s$ in $\Psi$}\label{tab6}
  \footnotesize 
  \begin{center}{\setlength{\tabcolsep}{0.5mm}
    \begin{tabular}{c| c|c|c|c|c|c|c|c|c|c}
         $O_s$    &$\begin{pmatrix} -1\\i\\1 \end{pmatrix}$&$ \begin{pmatrix} -1\\k\\1 \end{pmatrix}$&$ \begin{pmatrix} 0\\-1\\1 \end{pmatrix}$&$ \begin{pmatrix} 0\\0\\1 \end{pmatrix}$&
         $ \begin{pmatrix} 1\\-1\\1 \end{pmatrix}$&$ \begin{pmatrix}1\\1\\1 \end{pmatrix}$&$ \begin{pmatrix} 1\\i\\1 \end{pmatrix}$&$\begin{pmatrix} 1\\k\\0\end{pmatrix}$&
         $\begin{pmatrix} i\\k\\1 \end{pmatrix}$&$ \begin{pmatrix} k\\i\\1 \end{pmatrix}$ \\
       \hline
       $ \begin{pmatrix} -1\\i\\1 \end{pmatrix}   $&$ \begin{pmatrix} -1\\-k\\1 \end{pmatrix}$&$ \begin{pmatrix} -1\\1\\0\end{pmatrix}$&$\begin{pmatrix} -1\\0\\-j \end{pmatrix}$&
       $ \begin{pmatrix} 0\\0\\-i \end{pmatrix}$&$ \begin{pmatrix} -1\\1\\j \end{pmatrix}$&$ \begin{pmatrix} 1\\1\\-k \end{pmatrix}$&$ \begin{pmatrix} 0\\1\\i \end{pmatrix}$&
       $ \begin{pmatrix} 1\\0\\k \end{pmatrix}$&$ \begin{pmatrix} 1\\j\\1 \end{pmatrix}$&$ \begin{pmatrix} 0\\i\\1 \end{pmatrix}$ \\
       \hline
       $ \begin{pmatrix} -1\\k\\1 \end{pmatrix} $ && $ \begin{pmatrix} -1\\-i\\1 \end{pmatrix} $&$ \begin{pmatrix} -1\\0\\j \end{pmatrix} $&$ \begin{pmatrix} 0\\0\\-k \end{pmatrix} $&
       $ \begin{pmatrix} -1\\1\\-j \end{pmatrix} $&$ \begin{pmatrix} 1\\1\\-i \end{pmatrix} $&$\begin{pmatrix} 1\\0\\i \end{pmatrix} $&$ \begin{pmatrix} 0\\1\\k \end{pmatrix} $&
       $ \begin{pmatrix} 0\\k\\1 \end{pmatrix} $&$ \begin{pmatrix} 1\\-j\\1 \end{pmatrix}$ \\
       \hline
       $ \begin{pmatrix} 0\\-1\\1 \end{pmatrix} $ && & $\begin{pmatrix} 1\\-1\\1 \end{pmatrix} $&$ \begin{pmatrix} 0\\1\\0 \end{pmatrix} $&$ \begin{pmatrix} 0\\-1\\1 \end{pmatrix} $&
       $ \begin{pmatrix} -1\\-1\\1 \end{pmatrix} $&$ \begin{pmatrix} -1\\0\\i \end{pmatrix} $&$ \begin{pmatrix} -1\\0\\k \end{pmatrix} $&$ \begin{pmatrix} -1\\0\\-k \end{pmatrix} $&
       $ \begin{pmatrix} -1\\0\\-1 \end{pmatrix}$ \\
       \hline
       $ \begin{pmatrix} 0\\0\\1 \end{pmatrix} $ &&&& $ \begin{pmatrix} 0\\0\\1 \end{pmatrix} $&$\begin{pmatrix} -1\\0\\1 \end{pmatrix}$&$ \begin{pmatrix} 1\\0\\1 \end{pmatrix}$&
       $ \begin{pmatrix} 0\\0\\i \end{pmatrix} $&$ \begin{pmatrix} 0\\0\\k \end{pmatrix}$&$ \begin{pmatrix} 0\\0\\-j \end{pmatrix} $&$ \begin{pmatrix} 0\\0\\j \end{pmatrix}$\\
        \hline
        $\begin{pmatrix} 1\\-1\\1 \end{pmatrix} $ &&&&& $ \begin{pmatrix} 1\\1\\1 \end{pmatrix} $&$ \begin{pmatrix} 1\\1\\0 \end{pmatrix} $&$ \begin{pmatrix} -1\\1\\i \end{pmatrix} $&
        $ \begin{pmatrix} -1\\1\\k \end{pmatrix} $&$ \begin{pmatrix} -1\\1\\-i \end{pmatrix} $&$ \begin{pmatrix} -1\\1\\-k \end{pmatrix}$ \\
        \hline
        $\begin{pmatrix} 1\\1\\1 \end{pmatrix}$ &&&&&& $ \begin{pmatrix} 0\\1\\1 \end{pmatrix} $&$ \begin{pmatrix} 1\\1\\i \end{pmatrix} $&$ \begin{pmatrix} 1\\1\\k \end{pmatrix} $&
        $ \begin{pmatrix} 1\\1\\j \end{pmatrix} $&$ \begin{pmatrix} 1\\1\\-j \end{pmatrix}$\\
        \hline
        $ \begin{pmatrix} 1\\i\\0 \end{pmatrix}$ &&&&&&& $ \begin{pmatrix} -1\\-1\\i \end{pmatrix} $&$ \begin{pmatrix} 0\\0\\0 \end{pmatrix} $&$ \begin{pmatrix} 0\\-1\\i \end{pmatrix}$
        &$ \begin{pmatrix} 1\\-1\\i \end{pmatrix}$\\
        \hline
        $ \begin{pmatrix} 1\\k\\0 \end{pmatrix} $ &&&&&&&& $ \begin{pmatrix} -1\\-1\\k \end{pmatrix} $&$ \begin{pmatrix} 1\\-1\\k\end{pmatrix} $&$ \begin{pmatrix} 0\\-1\\k \end{pmatrix}$ \\
      \hline
       $ \begin{pmatrix} i\\k\\1 \end{pmatrix} $ &&&&&&&&& $ \begin{pmatrix} i\\1\\0 \end{pmatrix}$&$ \begin{pmatrix} -1\\1\\1 \end{pmatrix}$ \\
        \hline
        $\begin{pmatrix} k\\i\\1 \end{pmatrix}$ &&&&&&&&&& $ \begin{pmatrix} k\\1\\0 \end{pmatrix}$ 
    \end{tabular}}
  \end{center}
\end{table}

Using these tables, we will see that in $\Psi$ we can find pairs of ovals which carry an $n$-sided and an $m$-sided Poncelet polygon for $m \neq n$.
Indeed, we are able to find the $5$-sided Poncelet polygon
\small
$$ \begin{pmatrix} -1\\i\\1 \end{pmatrix}
\xrightarrow{(k,-j,1)} \begin{pmatrix} 0\\-1\\1 \end{pmatrix} 
\xrightarrow{(j,-k,1)} \begin{pmatrix} i\\k\\1 \end{pmatrix} 
\xrightarrow{(0,i,1)} \begin{pmatrix} 1\\i\\0 \end{pmatrix} 
\xrightarrow{(1,0,0)} \begin{pmatrix} 1\\k\\0 \end{pmatrix}
\xrightarrow{(j,k,1)} \begin{pmatrix} -1\\i\\1 \end{pmatrix}$$
\normalsize
and the $3$-sided Poncelet polygon
\small
$$ \begin{pmatrix} 0\\0\\1 \end{pmatrix}
\xrightarrow{(-1,-1,1)} \begin{pmatrix} 1\\1\\1 \end{pmatrix} 
\xrightarrow{(1,0,1)} \begin{pmatrix} 1\\-1\\1 \end{pmatrix}
\xrightarrow{(-1,1,1)} \begin{pmatrix} 0\\0\\1 \end{pmatrix}.$$
\normalsize
This shows that $\Psi$ is not a Poncelet plane.
\end{proof}

\section*{Acknowledgement}
We would like to thank Tim Penttila and the anonymous referee for helpful
comments on a previous version of this paper.


\end{document}